\documentclass[11pt]{amsart}
\usepackage{amsmath, amssymb}
\usepackage{amsfonts}
\usepackage[arrow,matrix,curve,cmtip,ps]{xy}

\usepackage{amsthm}

\allowdisplaybreaks

\newtheorem{theorem}{Theorem}[section]
\newtheorem{lemma}[theorem]{Lemma}
\newtheorem{proposition}[theorem]{Proposition}
\newtheorem{corollary}[theorem]{Corollary}

\newtheorem*{theorem*}{Theorem}
\theoremstyle{remark}
\newtheorem{remark}[theorem]{Remark}
\newtheorem{definition}[theorem]{Definition}
\newtheorem{example}[theorem]{Example}

\numberwithin{equation}{section}

\newcommand{\Z}{\mathbb{Z}}
\newcommand{\N}{\mathbb{N}}
\newcommand{\C}{\mathbb{C}}

\newcommand{\A}{\mathcal{A}}

\newcommand{\G}{\mathcal{G}}

\newcommand{\reg}{\textnormal{reg}}

\newcommand{\algspan}{\operatorname{span}}

\newcommand{\ch}{\operatorname{char}}

\begin{document}
\title{Leavitt path algebras with coefficients in a commutative ring}

\author{Mark Tomforde}

\address{Department of Mathematics \\ University of Houston \\ Houston, TX 77204-3008 \\USA}
\email{tomforde@math.uh.edu}

\thanks{The author was supported by NSA Grant H98230-09-1-0036.}

\date{May 4, 2009}

\subjclass[2000]{16W50, 46L55}

\keywords{graph algebras, rings, $R$-algebras, $C^*$-algebras}

\begin{abstract}
Given a directed graph $E$ we describe a method for constructing a Leavitt path algebra $L_R(E)$ whose coefficients are in a commutative unital ring $R$.  We prove versions of the Graded Uniqueness Theorem and Cuntz-Krieger Uniqueness Theorem for these Leavitt path algebras, giving proofs that both generalize and simplify the classical results for Leavitt path algebras over fields.  We also analyze the ideal structure of $L_R(E)$, and we prove that if $K$ is a field, then $L_K(E) \cong K \otimes_\Z L_\Z(E)$. 
\end{abstract}

\maketitle

\section{Introduction}

In \cite{AbrPino} the authors introduced a class of algebras over fields, which they constructed from directed graphs and called \emph{Leavitt path algebras}.  (The definition in \cite{AbrPino} was given for row-finite directed graphs, but the authors later extended the definition in \cite{AbrPino3} to all directed graphs.)  These Leavitt path algebras generalize the Leavitt algebras $L(1,n)$ of \cite{Leav2}, and also contain many other interesting classes of algebras over fields.  In addition, Leavitt path algebras are intimately related to graph $C^*$-algebras (see \cite{Rae}), and for any graph $E$ the Leavitt path algebra $L_\C(E)$ is $*$-isomorphic to a dense $*$-subalgebra of the graph $C^*$-algebra $C^*(E)$ \cite[Theorem~7.3]{Tom10}.

In this paper we generalize the construction of Leavitt path algebras by replacing the field $K$ with a commutative unital ring $R$.  We use the notation $L_R(E)$ for our Leavitt path algebra, and prove that it is a $\Z$-graded $R$-algebra with characteristic equal to the characteristic of $R$.  We also prove versions of the Graded Uniqueness Theorem and the Cuntz-Krieger Uniqueness Theorem, which are fundamental to the study of Leavitt path algebras.  

The Graded Uniqueness Theorem for Leavitt path algebras over a field says that a graded homomorphism $\phi : L_K(E) \to A$ is injective if $\phi(v) \neq 0$ for all $v \in E^0$.  For Leavitt path algebras over rings we need slightly different hypotheses: We prove that a graded homomorphism $\phi : L_R(E) \to A$ is injective if $\phi(rv) \neq 0$ for all $v \in E^0$ and for all $r \in R \setminus \{ 0 \}$.  Similarly, the Cuntz-Krieger Uniqueness Theorem for Leavitt path algebras over a field says that if every cycle in $E$ has an exit, then a homomorphism $\phi : L_K(E) \to A$ is injective if $\phi(v) \neq 0$ for all $v \in E^0$.  Again, our hypotheses for Leavitt path algebras over rings are slightly different: We prove that if every cycle in $E$ has an exit, then a homomorphism $\phi : L_K(E) \to A$ is injective if $\phi(rv) \neq 0$ for all $v \in E^0$ and for all $r \in R \setminus \{ 0 \}$.  

Our proofs of the Uniqueness Theorems use techniques that are different from those that have been used in the proofs for Leavitt path algebras over fields.  Consequently, this paper gives new proofs of each of the Uniqueness Theorems in the case that $R = K$ is a field.  One of the main points of this article is that our proofs of the Uniqueness Theorems are shorter than those in the existing literature.  

After proving our Uniqueness Theorems we continue by analyzing the ideal structure of $L_R(E)$.  For ease and clarity as we analyze ideals, we restrict our attention to the case when the graph $E$ is row-finite.  Because of the hypothesis $\phi(rv) \neq 0$ for all $v \in E^0$ and for all $r \in R \setminus \{ 0 \}$, the Uniqueness Theorems only allow us to analyze what we call basic ideals: an ideal $I$ of $L_R(E)$ is \emph{basic} if $rv \in I$ for $r \in R \setminus \{ 0 \}$ implies that $v \in I$.  In analogy with Leavitt path algebras over fields, we prove in Theorem~\ref{graded-ideals-structure} that the map $H \mapsto I_H$ is a lattice isomorphism from the saturated hereditary subsets of $E$ onto the graded basic ideals of $L_R(E)$.  We also prove in Theorem~\ref{K-iff-ideals-graded} that all basic ideals in $L_R(E)$ are graded if and only if $E$ satisfies Condition~(K).  Finally, in Theorem~\ref{bas-simp-cond-thm} and Proposition~\ref{simple-equiv-prop} we derive conditions for $L_R(E)$ to have no nontrivial proper basic ideals.  These results are similar to the classification of gauge-invariant ideals of graph $C^*$-algebras and Cuntz-Krieger $C^*$-algebras, and we use similar techniques in this paper.  We refer the reader to Remark~\ref{ideal-history} for references to the corresponding results for Cuntz-Krieger algebras, graph $C^*$-algebras, and Leavitt path algebras over fields.

In the final section, we discuss extending the coefficients of a Leavitt path algebra by tensoring with a commutative unital ring.  In particular, we show that if $K$ is a field, then $L_K(E) \cong K \otimes_\Z L_\Z(E)$; and if $K$ is a field of characteristic $p$, then $L_K(E) \cong K \otimes_{\Z_p} L_{\Z_p}(E)$.   This allows us to relate properties of $L_\Z(E)$ and $L_{\Z_p}(E)$ to properties of $L_K(E)$.

This paper is organized as follows:  After some preliminaries in \S \ref{prelim-sec}, we continue in \S \ref{constr-sec} by constructing the Leavitt path algebra over a commutative until ring, and prove that $L_R(E)$ exists and has the appropriate universal property.  In \S \ref{LPA-props-sec} we establish some basic properties of $L_R(E)$.  In \S \ref{GUT-sec} we prove the Graded Uniqueness Theorem for $L_R(E)$, and in \S \ref{CKUT} we prove the Cuntz-Krieger Uniqueness Theorem for $L_R(E)$.  In \S \ref{Ideals-sec} we analyze the ideal structure of $L_R(E)$.  Finally, in \S \ref{tensor-sec} we discuss extending the coefficients of a Leavitt path algebra by taking tensor products.  We conclude with a discussion of the significance of the rings $L_\Z(E)$ and $L_{\Z_n}(E)$.

\section{Preliminaries} \label{prelim-sec}

When we refer to a graph in this paper, we shall always mean a directed graph $E := (E^0, E^1, r, s)$ consisting of a countable set of vertices $E^0$, a countable set of edges $E^1$, and maps $r: E^1 \to E^0$ and $s:E^1 \to E^0$ identifying the range and source of each edge.

\begin{definition}
Let $E := (E^0, E^1, r, s)$ be a graph.  We say that a vertex $v \in E^0$ is a \emph{sink} if $s^{-1}(v) = \emptyset$, and we say that a vertex $v \in E^0$ is an \emph{infinite emitter} if $|s^{-1}(v)| = \infty$.  A \emph{singular vertex} is a vertex that is either a sink or an infinite emitter, and we denote the set of singular vertices by $E^0_\textnormal{sing}$.  We also let $E^0_\textnormal{reg} := E^0 \setminus E^0_\textnormal{sing}$, and refer to the elements of $E^0_\textnormal{reg}$ as \emph{regular vertices}; i.e., a vertex $v \in E^0$ is a regular vertex if and only if $0 < |s^{-1}(v)| < \infty$.
\end{definition}

\begin{definition} \label{graph-prelim-defs}
If $E$ is a graph, a \emph{path} is a sequence $\alpha := e_1 e_2 \ldots e_n$ of edges with $r(e_i) = s(e_{i+1})$ for $1 \leq i \leq n-1$.  We say the path $\alpha$ has \emph{length} $| \alpha| :=n$, and we let $E^n$ denote the set of paths of length $n$.  We consider the vertices in $E^0$ to be paths of length zero.  We also let $E^* := \bigcup_{n=0}^\infty E^n$ denote the paths of finite length, and we extend the maps $r$ and $s$ to $E^*$ as follows: For $\alpha := e_1 e_2 \ldots e_n \in E^n$, we set $r(\alpha) = r(e_n)$ and $s(\alpha) = s(e_1)$.   A \emph{cycle} in $E$ is a path $\alpha \in E^* \setminus E^0$ with $s(\alpha) = r(\alpha)$.  If $\alpha := e_1 \ldots e_n$, then an \emph{exit} for $\alpha$ is an edge $f \in E^1$ such that $s(f) = s(e_i)$ but $f \neq e_i$ for some $1 \leq i \leq n$.  We say that a graph $E$ satisfies Condition~(L) if every cycle in $E$ contains an exit. 
\end{definition}

\begin{definition}
We let $(E^1)^*$ denote the set of formal symbols $\{ e^* : e \in E^1 \}$, and for $\alpha = e_1 \ldots e_n \in E^n$ we define $\alpha^* := e_n^* e_{n-1}^* \ldots e_1^*$.  We also define $v^* = v$ for all $v \in E^0$.  We call the elements of $E^1$ \emph{real edges} and the elements of $(E^1)^*$ \emph{ghost edges}.
\end{definition}

\begin{definition} \label{Leavitt-E-fam-def}
Let $E$ be a directed graph and let $R$ be a ring.  A collection $\{ v, e, e^* : v \in E^0, e \in E^1 \} \subseteq R$ is a \emph{Leavitt $E$-family} in $R$ if $\{v : v \in E^0 \}$ consists of pairwise orthogonal idempotents and the following conditions are satisfied:
\begin{enumerate}
\item $s(e)e = er(e) =e$ for all $e \in E^1$
\item $r(e)e^* = e^* s(e) = e^*$ for all $e \in E^1$
\item $e^*f = \delta_{e,f} \, r(e)$ for all $e, f \in E^1$
\item $v = \displaystyle \sum_{\{e \in E^1 : s(e) = v \}} ee^*$ whenever $v \in E^0_\reg$.
\end{enumerate}
\end{definition}

\begin{definition} \label{Leavitt-def}
Let $E$ be a directed graph, and let $K$ be a field.  The \emph{Leavitt path algebra of $E$ with coefficients in $K$}, denoted $L_K(E)$,  is the universal $K$-algebra generated by a Leavitt $E$-family (see Definition~\ref{Leavitt-E-fam-def}).
\end{definition}

Note that $L_K(E)$ is universal for Leavitt $E$-families in $K$-algebras; i.e., if $A$ is a $K$-algebra and $\{ a_v, b_e, c_{e^*}: v \in E^0, e \in E^1 \}$ is a Leavitt $E$-family in $A$, then there exists a $K$-algebra homomorphism $\phi : L_K(E) \to A$ such that $\phi(v) = a_v$, $\phi(e)=b_e$, and $\phi(e^*)=c_{e^*}$ for all $v \in E^0$ and $e \in E^1$.  It is shown in \cite[\S 1]{AbrPino} and \cite[\S 1]{AbrPino3} that for any graph $E$ the generators $\{ v, e, e^* : v \in E^0, e \in E^1 \}$ of $L_K(E)$ are all nonzero.

In any algebra generated by a Leavitt $E$-family $\{ v, e, e^* : v \in E^0, e \in E^1 \}$, we see that 
\begin{equation} \label{mult-rules-eq}
(\alpha \beta^*)( \gamma \delta^*) = \begin{cases} \alpha \gamma' \delta^* & \text{ if $\gamma = \beta \gamma'$} \\ \alpha \delta^* & \text{ if $\beta = \gamma$} \\ \alpha \beta'^* \delta^* & \text{ if $\beta = \gamma \beta'$} \\ 0 & \text{ otherwise.} \end{cases}
\end{equation}

\subsection{Algebras over commutative rings}  If $R$ is a commutative ring with unit $1$, then an \emph{$R$-algebra} is an abelian group $A$ that has the structure of both a ring and a (left) $R$-module in such a way that 
\begin{itemize}
\item[(1)] $r \cdot (xy) = (r \cdot x) y = x (r \cdot y)$ for all $r \in R$ and $x,y \in A$; and
\item[(2)] $1\cdot x = x$ for all $x \in A$.
\end{itemize}
Note that as a ring, $A$ is not necessarily commutative and $A$ does not necessarily contain a unit.  By a  \emph{homomorphism between $R$-algebras} we mean an $R$-linear ring homomorphism.  If $A$ and $B$ are $R$-algebras, we let $\operatorname{Hom}_R (A,B)$ denote the collection of $R$-linear ring homomorphisms from $A$ to $B$.  We observe that for any $R$-algebra $A$, the endomorphism ring $\operatorname{Hom}_R (A,A)$ is an $R$-algebra in the obvious way.

If $R$ is a commutative ring, the characteristic of $R$, denoted $\ch (R)$, is defined to be the smallest positive integer $n$ such that $n r = 0$ for all $r \in R$, if such an $n$ exists, and $0$ otherwise. It is a fact that if $K$ is a field, then $\ch K$ is either equal to $0$ or a prime $p$.

Any ring $R$ may be viewed as a $\Z$-algebra in the natural way, and if $R$ has characteristic $n$, then $R$ may also be viewed as a $\Z_n$-algebra.  Furthermore, if $A$ is an $R$-algebra and $X \subseteq A$, then we define $$\algspan_R X := \left\{ \sum_{i=1}^n r_i x_i : r_i \in R \text{ and } x_i \in X \text{ for all } 1 \leq i \leq n \right\}$$ to be the $R$-submodule of $A$ generated by the set $X$.

\section{Constructing Leavitt path algebras with coefficients in a commutative ring with unit.} \label{constr-sec}

In this section we wish to extend the definition of a Leavitt path algebra to allow for coefficients in an arbitrary commutative ring with unit.

\begin{definition} \label{fun-ring-def}
Let $E$ be a directed graph, and let $R$ be a commutative ring with unit.   The \emph{Leavitt path algebra with coefficients in $R$}, denoted $L_R(E)$, is the universal $R$-algebra generated by a Leavitt $E$-family (see Definition~\ref{Leavitt-E-fam-def}).
\end{definition}

Note that $L_R(E)$ is universal for Leavitt $E$-families in $R$-algebras; i.e., if $A$ is a $R$-algebra and $\{ a_v, b_e, c_{e^*}: v \in E^0, e \in E^1 \}$ is a Leavitt $E$-family in $A$, then there exists a $R$-algebra homomorphism $\phi : L_R(E) \to A$ such that $\phi(v) = a_v$, $\phi(e)=b_e$, and $\phi(e^*)=c_{e^*}$ for all $v \in E^0$ and $e \in E^1$. 

Recall that any ring is a $\Z$-algebra and any ring of characteristic $n$ is a $\Z_n$-algebra.  This motivates the following definitions.

\begin{definition} \label{fundamental-ring-def}
If $E$ is a graph, the \emph{Leavitt path ring of characteristic $0$} is the ring $L_\Z(E)$, and for each $n \in \mathbb{N}$ the \emph{Leavitt path ring of characteristic $n$} is the ring $L_{\Z_n}(E)$.
\end{definition}

\begin{remark}
In the next proposition we show that the elements of $\{ v, e, e^* : v \in E^0, e \in E^1 \}$ are all nonzero, and that $rv \neq 0$ for all $v \in E^0$ and all $r \in R \setminus \{ 0 \}$.  In Proposition~\ref{lin-ind-prop}, we are able to prove a stronger result: The set of paths $E^*$ in $L_R(E)$ is linearly independent over $R$, and the set of ghost paths $\{ \alpha^* : \alpha \in E^* \}$ in $L_R(E)$ is linearly independent over $R$.
\end{remark}

The construction in the next proposition is an $R$-algebra version of a similar construction that has been done for graph $C^*$-algebras (see \cite[Theorem~1.2]{KPR}) and for Leavitt path algebras over fields (see \cite[Lemma~1.5]{Goo}).

\begin{proposition} \label{LPA-exists}
If $E$ is a graph and $R$ is a commutative ring with unit, then the Leavitt path algebra $L_R(E)$ has the property that the elements of the set $\{ v, e, e^* : v \in E^0, e \in E^1 \}$ are all nonzero.  Moreover, $$L_R(E) = \algspan_R \{ \alpha \beta^* : \alpha, \beta \in E^* \text{ and } r(\alpha) = r(\beta) \}$$ and $r v \neq 0$ for all $v \in E^0$ and all $r \in R \setminus \{ 0 \}$.
\end{proposition}

\begin{proof}
The fact that $e^*f = \delta_{e,f} r(e)$ allows us to write any word in the generators $\{ v, e, e^* : v \in E^0, e \in E^1 \}$ as $\alpha \beta^*$ with $\alpha, \beta \in E^*$.  It follows that $L_R(E) = \algspan_R \{ \alpha \beta^* : \alpha, \beta \in E^* \text{ and } r(\alpha) = r(\beta) \}$.

To see that the elements of the set $\{ v, e, e^* : v \in E^0, e \in E^1 \} \subseteq L_R(E)$ are all nonzero, it suffices (due to the universal property) to construct an $R$-algebra generated by nonzero elements satisfying the relations described in Definition~\ref{fun-ring-def}.  Define $Z :=R \oplus R \oplus \ldots$ to be the direct sum of countably many copies of $R$. For each $e \in E^1$ let $A_e := Z$, and for each $v \in E^0$ let $$A_v := \begin{cases} \displaystyle \bigoplus_{s(e)=v} A_e & \text{ if $0 < |s^{-1}(v) | < \infty$} \\ Z \displaystyle \oplus \bigoplus_{s(e)=v} A_e & \text{ if $|s^{-1}(v) | = \infty$} \\ Z & \text{ if $|s^{-1}(v) | = 0$.} \end{cases}$$  Note that the $A_v$'s and $A_e$'s are all mutually isomorphic since each is the direct sum of countably many copies of $R$.  Let $A := \bigoplus_{v \in E^0} A_v$.  For each $v \in E^0$ define $T_v : A_v \to A_v$ to be the identity map, and extend to a homomorphism $T_v : A \to A$ by defining $T_v$ to be zero on $A \ominus A_v$.  Also, for each $e \in E^1$ choose an isomorphism $T_e : A_{r(e)} \to A_e \subseteq A_{s(e)}$ and extend to a homomorphism $T_e : A \to A$ by defining $T_e$ to be zero on $A \ominus A_e$.  Finally, we define $T_{e^*} : A \to A$ by taking the isomorphism $T_e^{-1} : A_e \subseteq A_{s(e)} \to A_{r(e)}$ and extending to obtain a homomorphism $T_{e^*} : A \to A$ by defining $T_{e^*}$ to be zero on $A \ominus A_e$.  Let $A$ be the subalgebra of $\operatorname{Hom}_R (A, A)$ generated by $\{ T_v, T_e, T_{e^*} : v \in E^0, e \in E^1 \}$. One can check that $\{ T_v, T_e, T_{e^*} : v \in E^0, e \in E^1 \}$ is a collection of nonzero elements satisfying the relations described in Definition~\ref{fun-ring-def}.  Thus the subalgebra of $\operatorname{Hom}_R (A, A)$ generated by $\{ T_v, T_e, T_{e^*} : v \in E^0, e \in E^1 \}$ is the desired $R$-algebra.

Finally, we note that for any $v$ we have $A_v = R \oplus M$ for some $R$-module $M$.  Thus for any $r \in R \setminus \{ 0 \}$, using the fact that $R$ is unital we have $r T_v(1,0) = T_v(r,0) = (r,0) \neq 0$.  Hence $rT_v \neq 0$.  The universal property of $L_R(E)$ then implies that $rv \neq 0$ for any $v \in E^0$ and any $r \in R \setminus \{ 0 \}$.  
\end{proof}

\begin{corollary}
Let $E$ be a graph and let $R$ be a commutative ring with unit. Then $\ch L_R(E) = \ch R$.
\end{corollary}

\begin{remark}[A realization of $L_R(E)$] \label{Alt-construction-L(E)}
Suppose $E$ is a graph and $R$ is a commutative ring with unit.  The \emph{path algebra of $E$ with coefficients in $R$} is the $R$-algebra generated by paths with the operation of path concatenation. (Here vertices are considered as paths of length zero.)  In other words, $A_R(E)$ is the free $R$-algebra generated by the paths $E^* = \bigcup_{n=0}^\infty E^n$ with the following relations:
\begin{itemize}
\item[(i)]  $v w = \delta_{v,w} v$ for all $v, w \in E^0$
\item[(ii)] $e = er(e) = s(e)e$ for all $e \in E^1$.
\end{itemize}
If $E = (E^0, E^1, r, s)$ is a graph, we let $\hat{E}$ be the graph with vertex set $\hat{E}^0 := E^0$, edge set $\hat{E}^1 := \{e, e^* : e \in E^1 \}$, and maps $r$ and $s$ extended to $\hat{E}^1$ by $r(e^*) := s(e)$ and $s(e^*) = r(e)$ for all $e \in E^1$.  We see that $L_R(E)$ may be realized as 
the quotient $A_R(\hat{E}) / I$, where $A_R(\hat{E})$ is the path algebra of $\hat{E}$ with coefficients in $R$, and $I$ is the ideal of $A_R(\hat{E})$ generated by the elements
\begin{equation} \label{CK-relations-ideal}
\left\{ e^*f - \delta_{e,f} r(e) : e, f \in E^1 \right\} \cup \big\{ v - \sum_{s(e)=v} ee^* : v \in E^0_\reg \big\}.
\end{equation}
\end{remark}

\section{Properties of Leavitt Path Algebras} \label{LPA-props-sec}

\subsection{Involution and selfadjoint ideals}

As we have seen, any element $x \in L_R(E)$ may be written $x = \sum_{k=1}^N r_k \alpha_k \beta_k^*$ where $\alpha_k, \beta_k \in E^*$ with $r(\alpha_k) = r(\beta_k)$ and $r_k \in R$ for $1 \leq k \leq N$.

\begin{remark}
If $E$ is a graph, $R$ is a commutative ring with unit, and $L_R(E)$ is the associated Leavitt path algebra, we may define an $R$-linear involution $x \mapsto x^*$ on $L_R(E)$ as follows:  If $x = \sum_{k=1}^N r_k \alpha_k \beta_k^*$, then $x^* = \sum_{k=1}^N r_k \beta_k \alpha_k^*$.  Note that this operation is $R$-linear, involutive ($(x^*)^* = x$), and antimultiplicative ($(xy)^* = y^*x^*$).
\end{remark}

\begin{definition}
If $L_R(E)$ is the Leavitt path algebra of a graph $E$ with coefficients in $R$, an ideal $I$ of $L_R(E)$ is \emph{selfadjoint} if $I^* = I$.
\end{definition}

\subsection{Enough idempotents and local units}

A ring $R$ has \emph{enough idempotents} if there exists a collection of pairwise orthogonal idempotents $\{ e_\alpha \}_{\alpha \in \Lambda}$ such that $R = \bigoplus_{\alpha \in \Lambda} e_\alpha R = \bigoplus_{\alpha \in \Lambda} R e_\alpha$.  A \emph{set of local units} for a ring $R$ is a set $\Lambda \subseteq R$ of commuting idempotents with the property that for any $x \in R$ there exists $t \in \Lambda$ such that $tx=xt=x$.  

If $E$ is a graph, $R$ is a commutative ring with unit, and $L_R(E)$ is the associated Leavitt path algebra, then $$L_R(E) = \bigoplus_{v \in E^0} v L_R(E) = \bigoplus_{v \in E^0} L_R(E) v$$ so $L_R(E)$ is a ring with enough idempotents.  Furthermore, if $E^0$ is finite, then $1=\sum_{v \in E^0} v$ is a unit for $L_R(E)$.  If $E^0$ is infinite, then $L_R(E)$ does not have a unit, but if we list the vertices of $E$ as $E^0 = \{v_1, v_2, \ldots \}$ and set $t_n := \sum_{k=1}^n v_k$, then $\{ t_n \}_{n \in \N}$ is a set of local units for $L_R(E)$.

\begin{definition}
A ring $R$ is \emph{idempotent} if $R^2 = R$; that is, if every $x \in R$ can be written as $x = \sum_{k=1}^n a_k b_k$ for $a_1, \ldots a_n, b_1, \ldots, b_n \in R$.
\end{definition}

\begin{remark}
We see that if $R$ is a ring with a set of local units, then $R$ is idempotent:  If $x \in R$, then there exists an idempotent $t \in R$ with $x = tx$.  Consequently, the Leavitt path algebra $L_R(E)$ is an idempotent ring.
\end{remark}

\subsection{$\Z$-graded rings}

We show that all Leavitt  path algebras have a natural $\Z$-grading.  

\begin{definition}
If $R$ is a ring, we say $R$ is \emph{$\Z$-graded} if there is a a collection of additive subgroups $\{ R_k \}_{k \in \Z}$ of $R$ with the following two properties:
\begin{enumerate}
\item $R = \bigoplus_{k \in \Z} R_k$
\item $R_j R_j \subseteq R_{j+k}$ for all $j,k \in \Z$.
\end{enumerate}
The subgroup $R_k$ is called the \emph{homogeneous component of $R$ of degree $k$}.
\end{definition}

\begin{definition}
If $R$ is a graded ring, then an ideal $I$ of $R$ is a \emph{$\Z$-graded ideal} if $I = \bigoplus_{k \in \Z} (I \cap R_k)$.  If $\phi : R \to S$ is a ring homomorphism between $\Z$-graded rings, then $\phi$ is a \emph{graded ring homomorphism} if $\phi(R_k) \subseteq S_k$ for all $n \in \Z$.
\end{definition}

Note that the kernel of a $\Z$-graded homomorphism is a $\Z$-graded ideal.  Also, if $I$ is a $\Z$-graded ideal in a $\Z$-graded ring $R$, then the quotient $R / I$ admits a natural $\Z$-grading and the quotient map $R \to R /I$ is a $\Z$-graded homomorphism.  In this paper we will be concerned only with $\Z$-gradings, and hence we will often omit the prefix $\Z$ and simply refer to rings, ideals, homomorphisms, etc. ~as \emph{graded}. 

\begin{proposition} \label{L(E)-Z-graded}
If $E$ is a graph and $R$ is a commutative ring with unit, then we may define a $\Z$-grading on the associated Leavitt path algebra $L_R(E)$ by setting $$L_R(E)_k := \left\{ \sum_{i=1}^N r_i \alpha_i \beta_i^* : \alpha_i, \beta_i \in E^*, r_i \in R, \text{ and } |\alpha_i | - | \beta_i| = k \text{ for all $i$} \right\}.$$
\end{proposition}

\begin{proof}
Let $A$ be the free $R$-algebra generated by $E^0 \cup E^1 \cup (E^1)^*$.  Then $A$ has a unique $\Z$-grading for which the elements of $E^0$, $E^1$, and $(E^1)^*$ have degrees $0$, $1$, and $-1$, respectively.  Let $I$ be the ideal in $A$ generated by elements of the following type:
\begin{itemize}
\item $vw - \delta_{v,w} v$ for $v, w \in E^0$
\item $e-er(e)$ for $e \in E^1$
\item $e-s(e)e$ for $e \in E^1$
\item $e^*f - \delta_{e,f} r(e)$ for $e,f \in E^1$
\item $v - \sum_{s(e)=v} ee^*$ for $v \in E^0_\textnormal{reg}$.
\end{itemize}
Since the elements generating $I$ are all homogeneous of degree zero, it follows that $I$ is a graded ideal.  Furthermore, we see that $A/I \cong L_R(E)$, so that $L_R(E)$ is graded with the homogeneous elements of degree $k$ equal to the set of $R$-linear combinations of elements of the form $\alpha \beta^*$ with $|\alpha | - | \beta| = k$.
\end{proof}

\begin{definition}
If $x \in L_R(E)$, we say that $x$ is a \emph{polynomial in real edges} if $x = \sum_{i=1}^n r_i \alpha_i$ for $r_i \in R \setminus \{ 0 \}$ and $\alpha_i \in E^*$.  In this case we also define the \emph{degree of $x$} to be $$\deg x = \max \{ | \alpha_i| : 1 \leq i \leq n \}.$$  Note that $\deg x$ is independent of how $x$ is written.
\end{definition}

\begin{proposition} \label{lin-ind-prop}
Let $E$ be a graph and let $R$ be a commutative ring with unit.  The set of paths $E^*$ in $L_R(E)$ is linearly independent over $R$.  Likewise, the set of ghost paths $\{ \alpha^* : \alpha \in E^* \}$ in $L_R(E)$ is linearly independent over $R$. 
\end{proposition}

\begin{proof}
Suppose that $\alpha_1, \ldots, \alpha_n \in E^*$, and $\sum_{i=1}^n r_i \alpha_i = 0$ for some $r_1, \ldots, r_n \in R$.  Using the $\Z$-grading on $L_R(E)$ we may, without loss of generality, assume that all the $\alpha_i$'s have the same length.  Then for any $1 \leq j \leq n$ we have $r_j (\alpha_j) = \alpha_j^* \alpha_j = \alpha_j^* ( \sum_{i=1}^n r_i \alpha_i) = 0$.  Proposition~\ref{LPA-exists} implies that $r_i = 0$.  It follows that  $\{ \alpha_1, \ldots, \alpha_n \}$ is linearly independent over $R$.  A similar argument works for ghost paths.
\end{proof}

\subsection{Morita equivalence}

Throughout this paper we will need to discuss Morita equivalence for rings that do not necessarily have an identity element.  We recall the necessary definitions and results here.

\begin{definition}
If $R$ is a ring, we say that a left $R$-module $M$ is \emph{unital} if $RM = M$.  We also say that $M$ is \emph{nondegenerate} if for all $m \in M$ we have that $Rm = 0$ implies that $m = 0$.  We let $R$-MOD denote the full subcategory of the category of all $R$-modules whose objects are unital nondegenerate $R$-modules.  (Note that if $R$ is unital, $R$-MOD is the usual category of $R$-modules.)  When $R$ and $S$ are rings, and ${}_RM_S$ is a bimodule, we say $M$ is \emph{unital} if $RM = M$ and $MS=M$.
\end{definition}

\begin{definition}
Let $R$ and $S$ be idempotent rings.  A \emph{(surjective) Morita context} $(R,S,M,N,\psi, \phi)$ between $R$ and $S$ consists of unital bimodules ${}_RM_S$ and ${}_SN_R$, a surjective $R$-module homomorphism $\psi : M \otimes_S N \to R$, and a surjective $S$-module homomorphism $\phi : N \otimes_R M \to S$ satisfying
$$\phi (n \otimes m) n' = n \psi (m \otimes n') \qquad \text{ and } \qquad m' \phi(n \otimes m) = \psi (m' \otimes n)m $$ for every $m, m' \in M$ and $n,n' \in N$.  We say that $R$ and $S$ are \emph{Morita equivalent} in the case that there exists a Morita context.
\end{definition}

It is proven in \cite[Proposition~2.5]{GS} and \cite[Proposition~2.7]{GS} that $R$-MOD and $S$-MOD are equivalent categories if and only if there exists a Morita context $(R,S,M,N,\psi, \phi)$.  In addition, the following result is obtained in \cite{GS}. 

\begin{proposition}  \cite[Proposition~3.5]{GS} \label{ideal-corresp-Mor-eq}
Let $R$ and $S$ be Morita equivalent idempotent rings, and let $(R,S,M,N,\psi, \phi)$ be a Morita context.  If $$\mathcal{L}_R := \{ I \subseteq R : \text{ $I$ is an ideal and $RIR=I$} \}$$ and $$\mathcal{L}_S := \{ I \subseteq S : \text{ $I$ is an ideal and $SIS=I$} \},$$ then there is a lattice isomorphism from $\mathcal{L}_R$ onto $\mathcal{L}_S$ given by $I \mapsto \phi (NI, M)$ with inverse given by $I \mapsto \psi (MI, N)$.
\end{proposition}

\begin{remark} \label{ideal-corresp-rings-loc-units}
Note that when $R$ is a ring with a set of local units, $\mathcal{L}_R$ is the lattice of ideals of $R$.  Thus if each of $R$ and $S$ is a ring with a set of local units, and if $R$ and $S$ are Morita equivalent, then the lattice of ideals of $R$ is isomorphic to the lattice of ideals of $S$.
\end{remark}

Recall that in rings the property of being a ring ideal is not transitive; i.e., if $R$ is a ring, $I$ is an ideal of $R$, and $J$ is an ideal of $I$, then it is not necessarily true that $J$ is an ideal of $R$.  Despite this fact, there is a special case when the implication does hold, and this will be of use to us.

\begin{lemma} \label{local-units-imply-trans}
Let $R$ be a ring and let $I$ be an ideal of $R$ with the property that $I$ has a set of local units.  If $J$ is an ideal of $I$, then $J$ is an ideal of $R$.
\end{lemma}

\begin{proof}
Let $r \in R$ and $x \in J$.  Since $I$ has a set of local units, there exists $t \in I$ with $tx = x$.  Because $I$ is an ideal, we have that $rt \in I$.  Hence $rx = r(tx) = (rt) x \in J$.  A similar argument shows that $xr \in I$.
\end{proof}

\section{The Graded Uniqueness Theorem} \label{GUT-sec}

\begin{lemma} \label{ideal-gen-by-zero-part}
Let $I$ be a graded ideal of $L_R(E)$.  Then $I$ is generated as an ideal by the set $I_0 := I \cap L_R(E)_0$. 
\end{lemma}

\begin{proof}
Let $k >0$.  Given $x \in I_k :=I \cap  L_R(E)_k$, we may write $x = \sum_{i=1}^n \alpha_i x_i$, where $x_i \in L_R(E)_0$ and $\alpha_i \in E^k$ for all $1 \leq i \leq n$, and $\alpha_i \neq \alpha_j$ for $i \neq j$.  Then for any $1 \leq j \leq n$ we have $$x_j = \alpha_j^* \left(  \sum_{i=1}^n \alpha_i x_i \right) = \alpha_j^* x \in I.$$  Thus $x_j \in I_0$ and $I_k = L_R(E)_k I_0$.  Similarly, $I_{-k} = I_0 L_R(E)_{-k}$.  Since $I$ is a graded ideal, $I = \bigoplus_{k \in \Z} I_k$, and $I$ is generated as an ideal by $I_0$.
\end{proof}

\begin{lemma} \label{zero-part-zero}
Let $E$ be a graph, and let $R$ be a commutative ring with unit.  If $x \in L_R(E)_0$ and $x \neq 0$, then there exists $\alpha, \beta \in E^*$ such that $\alpha^* x \beta = rv$ for some $v \in E^0$ and some $r \in R \setminus \{ 0 \}$.
\end{lemma}

\begin{proof}
Define $\G_N := \algspan_{R} \{ \alpha \beta^* : \alpha, \beta \in E^m \text{ for } 1 \leq m \leq N \}$.  Then $L_R(E)_0 = \bigcup_{N=0}^\infty \G_N$.  We will prove by induction on $N$ that if $x \in \G_N$ and $x \neq 0$, then there exists $\alpha, \beta \in E^*$ such that $\alpha^* x \beta = rv$ for some $v \in E^0$ and some $r \in R \setminus \{ 0 \}$.  In the base case we have $N=0$, and $x = \sum_{i=1}^n r_i v_i$ for $v_i \in E^0$ and nonzero $r_i \in R$ with $v_i \neq v_j$ for $i \neq j$.  If we let $\alpha = \beta = v_1$, then $\alpha^* x \beta = r_1 v_1$.  

In the inductive step, we assume that for all nonzero $y \in \G_{N-1}$ there exists $\alpha', \beta' \in E^*$ such that $(\alpha')^* y \beta' = rv$ for some $v \in E^0$ and some $r \in R \setminus \{ 0 \}$.  Suppose that $x \in \G_N$ and $x \neq 0$.  Then we can write $$x = \sum_{i=1}^M r_i \alpha_i \beta_i^* + \sum_{j=1}^P s_j v_j,$$ for $\alpha, \beta \in E^*$ with $| \alpha_i | = | \beta_i | \geq 1$, $v_j \in E^0$ with $v_j \neq v_{j'}$ for $j \neq j'$, and $r_i, s_j \in R \setminus \{ 0 \}$.  If any $v_j$ is a sink, we may let $\alpha = \beta = v_j$, and then $\alpha^* x \beta = s_j v_j$.  If any $v_j$ is an infinite emitter, then we may choose an edge $e \in E^1$ with $s(e) = v_j$ and $e$ not equal to any edge appearing in any of the $\alpha_i$'s.  If we let $\alpha = \beta = e$, then $\alpha^* x \beta = e^* s_j v_j e = s_j r(e)$.  The only other case to consider is when every $v_j$ is a regular vertex (i.e., neither a sink nor an infinite emitter).  In this case we may use the relation $v_j = \sum_{s(e)=v_j} ee^*$ to write $x$ as a linear combination of elements $\gamma \delta^*$ where $\gamma, \delta \in E^*$ with $|\gamma| = |\delta | \geq 1$.  By regrouping the elements in this linear combination, we may write $$x = \sum_{i=1}^P \sum_{j=1}^Q e_i x_{i,j} f_j^*$$ where $e_i, f_i \in E^1$ with $e_i \neq e_{i'}$ for $i \neq i'$ and $f_j \neq f_{j'} $ for $j \neq j'$; and $x_{i,j} \in \G_{N-1}$ with $e_i x_{i,j} f_j^* \neq 0$ for all $i,j$.    Since $e_1 x_{1,1} f_1^* \neq 0$, it follows that $r(e_1) x_{1,1} r(f_1) \neq 0$.  Because $r(e_1) x_{1,1} r(f_1) \neq 0$ and $r(e_1) x_{1,1} r(f_1) \in \G_{N-1}$, the inductive hypothesis implies that there exists $\alpha', \beta' \in E^*$ such that $(\alpha')^* r(e_1) x_{1,1} r(f_1) \beta' = rv$ for some $v \in E^0$ and some $r \in R \setminus \{ 0 \}$.  If we let $\alpha := e_1 \alpha'$ and $\beta := f_1 \beta'$, then $$\alpha^* x \beta = (\alpha')^* e_1^* x  f_1 \beta' = (\alpha')^* e_1^*e_1 x_{1,1} f_1^* f_1 \beta' = (\alpha')^* r(e_1) x_{1,1} r(f_1) \beta' = rv.$$  The Principle of Mathematical Induction shows that the claim holds for all $N$, and hence the lemma holds for all nonzero $x$ in $L_R(E)_0$.
\end{proof}

\begin{theorem}[Graded Uniqueness Theorem] \label{GUT}
Let $E$ be a graph, and let $R$ be a commutative ring with unit.  If $S$ is a graded ring and $\phi : L_R(E) \to S$ is a graded ring homomorphism with the property that $\phi(rv) \neq 0$ for all $v \in E^0$ and for all $r \in R \setminus \{ 0 \}$, then $\phi$ is injective.
\end{theorem}

\begin{proof}
Suppose that $x \in L_R(E)_0 \cap \ker \phi$.  If $x$ is nonzero, then by Lemma~\ref{zero-part-zero} there exists $\alpha, \beta \in E^*$ such that $\alpha^* x \beta = rv$ for some $v \in E^0$ and some $r \in R \setminus \{ 0 \}$.  But then $\phi(rv) = \phi(\alpha^* x \beta) = \phi(\alpha^*) \phi(x) \phi(\beta) = 0$, which is a contradiction.  Hence $x = 0$, and $L_R(E)_0 \cap \ker \phi = \{ 0 \}$.

Since $\phi$ is a graded ring homomorphism, $\ker \phi$ is a graded ideal of $L_R(E)$.  It follows from Lemma~\ref{ideal-gen-by-zero-part} that $\ker \phi$ is generated as an ideal by $L_R(E)_0 \cap \ker \phi = \{ 0 \}$.  Thus $\ker \phi = \{ 0 \}$, and $\phi$ is injective.
\end{proof}

\begin{corollary}
Let $E$ be a graph, and let $K$ be a field.  If $S$ is a graded ring and $\phi : L_K(E) \to S$ is a graded ring homomorphism with the property that $\phi(v) \neq 0$ for all $v \in E^0$, then $\phi$ is injective.
\end{corollary}

\begin{remark} \label{GUT-better-rem}
The Graded Uniqueness Theorem for Leavitt path algebras may be thought of as an analogue of the Gauge-Invariant Uniqueness Theorem for graph $C^*$-algebras, with the grading playing the role of the gauge action.      

In \cite{CK} Cuntz and Krieger showed that if $A$ is a finite $\{0,1\}$-matrix satisfying Condition~(I), then there is a unique $C^*$-algebra generated by a nonzero Cuntz-Krieger $A$-family, which they denote by $\mathcal{O}_A$.  Universal Cuntz-Krieger algebras of finite $\{0,1\}$-matrices were introduced in \cite{aHR}, and a Gauge-Invariant Uniqueness Theorem for these algebras was proven in \cite[Theorem~2.3]{aHR}.  A Gauge-Invariant Uniqueness Theorem for $C^*$-algebras of row-finite graphs was obtained in \cite[Theorem~2.1]{BPRS}, and this was extended to $C^*$-algebras of non-row-finite graphs in \cite[Theorem~2.1]{BHRS}.  Furthermore, the Gauge-Invariant Uniqueness Theorem was generalized to Cuntz-Krieger algebras of infinite matrices in \cite[Theorem~2.7]{RS} and to Cuntz-Pimnser algebras in \cite[Theorem~4.1]{FMR} and \cite[Theorem~6.4]{Kat4}.

In the Leavitt path algebra setting, the gauge action is replaced by a $\mathbb{Z}$-grading --- in fact, if one views the Leavitt path algebra $L_\C(E)$ as a dense $*$-subalgebra of the graph $C^*$-algebra $C^*(E)$, then the gauge action on $C^*(E)$ induces a $\mathbb{Z}$-grading on $L_\C(E)$ (see the proof of \cite[Theorem~7.3]{Tom10} for details).  In \cite[Theorem~5.1]{AMP}, Ara, Moreno, and Pardo proved the Graded Uniqueness Theorem for $L_K(E)$, where $K$ is a field and $E$ is a row-finite graph.  A proof of the Graded Uniqueness Theorem for $L_K(E)$, where $K$ is a field and $E$ is an arbitrary graph, was given by the author in \cite[Theorem~4.8]{Tom10}.  The proof in Theorem~\ref{GUT} uses different techniques than \cite{AMP} or \cite{Tom10}.
\end{remark}

\section{The Cuntz-Krieger Uniqueness Theorem} \label{CKUT}

Recall that a graph $E$ is said to satisfy Condition~(L) if every cycle in $E$ has an exit.  (See Definition~\ref{graph-prelim-defs} for more details.)

\begin{lemma} \label{nonreturning-path-lemma}
Suppose $E$ is a graph satisfying Condition~(L).  If $F$ is a finite subset of $E^* \setminus E^0$ and $v \in E^0$, then there exists a path $\alpha \in E^*$ such that $s(\alpha)=v$ and for every $\mu \in F$ we have $\alpha^* \mu \alpha = 0$.
\end{lemma}

\begin{proof}
Given $v \in E^0$ and a finite subset $F \subseteq E^*$, consider two cases.

\noindent \textsc{Case I:} There is a path from $v$ to a sink in $E$.
In this case, let $\alpha$ be a path with $s(\alpha) = v$ and $r(\alpha)$ a sink.  For any $\mu \in F$, we see that $\alpha^* \mu \alpha$ is nonzero if and only if there exists $\nu \in E^* \setminus E^0$ such that $\mu \alpha = \alpha \nu$, which is impossible since $r(\alpha)$ is a sink.  Thus $\alpha^* \mu \alpha = 0$.

\noindent \textsc{Case II:} There is no path from $v$ to a sink in $E$.

Let $M = \max \{ | \mu | : \mu \in F \} + 1$.  If there is a path $\alpha = \alpha_1 \ldots \alpha_M \in E^M$ with $s(\alpha) = v$ and no repeated vertices, then for any $\mu \in F$ we see that $\alpha^* \mu \alpha$ is nonzero if and only if there exists $\nu \in E^* \setminus E^0$ such that $\mu \alpha = \alpha \nu$, which is impossible since this would imply that $s(\alpha_1) = s(\alpha_j)$ for some $j \geq 2$ contradicting that $\alpha$ has no repeated vertices.  Thus $\alpha^* \mu \alpha = 0$.

Otherwise, every path $E^M$ with $s(\lambda) = v$ has repeated vertices, and there exists a path from $v$ to the base point of a cycle in $E$.  Choose a path $\tau$ of minimal length such that $s(\tau) = v$ and $r(\tau)$ is the base point of a cycle.  Choose a cycle $\beta$ of minimal length based at $r(\tau)$.  Let $f$ be an exit for $\beta$, and let $\beta'$ be the segment of $\beta$ from $r(\tau)$ to $s(f)$.  By the minimality of $\tau$, the edge $f$ is not equal to any of the edges in the path $\tau$.  Likewise, by the minimality of $\beta$, the edge $f$ is not equal to any of the edges on the cycle $\beta$ or the path $\beta'$.  Thus the path $\alpha := \tau \beta \beta \ldots \beta \beta' f$ has the property that $f$ is not equal to any edge $\alpha_i$ for $1 \leq i \leq |\alpha|-1$.  By choosing sufficiently many repetitions of the cycle $\beta$ we can ensure that $\alpha$ has length greater than or equal to $M$ (to avoid the possibility that $\alpha \in F$).  Then we have that $\alpha^* \mu \alpha$ is nonzero if and only if there exists $\nu \in E^* \setminus E^0$ such that $\mu \alpha = \alpha \nu$, which is impossible since this would imply that $f = \alpha_j$ for some $1 \leq j \leq |\alpha|-1$.  Thus $\alpha^* \mu \alpha = 0$.
\end{proof}

\begin{lemma} \label{shift-real-to-vertex-lem}
Let $E$ be a graph satisfying Condition~(L), and let $R$ be a commutative ring with unit.  If $x \in L_R(E)$ is a polynomial in only real edges and $x \neq 0$, then there exist paths $\alpha, \beta \in E^*$ such that $\alpha^* x \beta = rv$ for some $v \in E^0$ and some $r \in R \setminus \{ 0 \}$.
\end{lemma}

\begin{proof}
We will prove by induction on $N$ that if $x \in L_R(E)$ is a nonzero polynomial in only real edges with $\operatorname{deg} x = N$, then there exist paths $\alpha, \beta \in E^*$ such that $\alpha^* x \beta = rv$ for some $v \in E^0$ and some $r \in R \setminus \{ 0 \}$.  In the base case we have $\operatorname{deg} x = 0$, so that $x = \sum_{i=1}^M r_i v_i$ for $v_i \in E^0$ and nonzero $r_i \in R$ with $v_i \neq v_j$ for $i \neq j$.  If we let $\alpha = \beta = v_1$, then $\alpha^* x \beta = r_1 v_1$.

In the inductive step, we assume that our claim holds for all nonzero polynomials in real edges with degree $N-1$ or less.  Suppose $x \in L_R(E)$ is a nonzero polynomial in real edges with $\operatorname{deg} x = N$.  If $x$ has no terms of degree 0, then we may write $$x = \sum_{i=1}^M e_i x_i $$ with each $x_i$ a nonzero polynomial in real edges of degree $N-1$ or less, and $e_i \in E^1$ with $e_i \neq e_j$ for $i \neq j$.  Then $e_1^* x = x_1$ is a nonzero polynomial of degree $N-1$ or less, so by the inductive hypothesis there exists $\alpha', \beta \in E^*$ such that $(\alpha')^* x_1 \beta = rv$ for some $v \in E^0$ and $r \in R \setminus \{ 0 \}$.  If we let $\alpha := e_1 \alpha'$, then $\alpha^* x \beta = (\alpha')^* e_1^*x \beta = (\alpha')^* x_1 \beta =  rv$ and the claim holds.  On the other hand, if $x$ does have a term of degree 0, then we may write $$x = \sum_{i=1}^M r_i \alpha_i  + \sum_{j=1}^K s_j v_j$$ where the $\alpha_i$'s are paths of length 1 or greater, each $r_i, s_j \in R \setminus \{ 0 \}$, and the $v_j$'s are vertices with $v_j \neq v_{j'}$ for $j \neq j'$.  Let $F := \{ \alpha_i : 1 \leq i \leq M \}$.  By Lemma~\ref{nonreturning-path-lemma} there exists $\alpha \in E^*$ such that $s(\alpha)=v_1$ and for every $\alpha_i$ we have $\alpha^* \alpha_i \alpha = 0$.  If we let $\beta= \alpha$, then we have $$\alpha^* x \beta = \sum_{i=1}^M r_i \alpha^* \alpha_i \alpha + \sum_{j=1}^K s_j \alpha^* v_j \alpha = s_1 \alpha^* v_1 \alpha = s_1 r(\alpha).$$  By the Principle of Mathematical Induction, we may conclude that the claim holds for all $N$, and hence the lemma holds for all nonzero polynomials in only real edges.
\end{proof}

\begin{lemma} \label{xe-not-zero-lem}
Let $E$ be a graph and let $R$ be a commutative ring with unit.  Let $x \in L_R(E)$ and suppose that $x$ is a polynomial in real edges with $x \neq 0$.  If there exists $v \in E^0$ with $xv = x$, then for any $e \in E^1$ with $s(e) = v$ we have $xe \neq 0$.
\end{lemma}

\begin{proof}
Since $L_R(E)$ is graded with $L_R(E) = \bigoplus_{k \in \Z} L_R(E)_k$, it suffices to prove the claim when $x$ is homogeneous of degree $k$ for some $k \geq 0$.  In this case we may write $x = \sum_{i=1}^M r_i \alpha_i$ with each $r_i \in R \setminus \{ 0 \}$ and each $\alpha_i \in E^k$ with $\alpha_i \neq \alpha_{i'}$ for $i \neq i'$.  Since $xv = x$, we may also assume that $r(\alpha_i) = v$ for all $i$.  For any $e \in E^1$ with $s(e) = v$ we see that $\alpha_i e \in E^{k+1}$.  If $xe = 0$, then $$r_1 r(e) = e^* \alpha_1^* (r_1 \alpha_1 e) = e^* \alpha_1^* \left( \sum_{i=1}^M r_i \alpha_i \right) e = e^* \alpha_1^* (xe) = 0,$$ which contradicts Proposition~\ref{LPA-exists}.  Hence $xe \neq 0$.
\end{proof}

\begin{lemma} \label{shift-to-real-lem}
Let $E$ be a graph and let $R$ be a commutative ring with unit.  If $x \in L_R(E)$ and $x \neq 0$ then there exists $\gamma \in E^*$ such that $x \gamma \neq 0$ and $x \gamma$ is a polynomial in only real edges.
\end{lemma}

\begin{proof}
Define $$\A_N :=  \left\{ \sum_{i=1}^M r_i \alpha_i \beta_i^* :  r_i \in R, \ \alpha_i, \beta_i \in E^*, \text{ and }  | \beta_i | \leq N \text{ for } 1 \leq i \leq M \right\}.$$  Then $L_R(E) = \bigcup_{N=0}^\infty \A_N$.  We will prove by induction on $N$ that if $x \in \A_N$ and $x \neq 0$, then there exists $\gamma \in E^*$ such that $x \gamma \neq 0$ and $x \gamma$ is a polynomial in only real edges.  In the base case we have $N=0$, and $x = \sum_{i=1}^M r_i \alpha_i$ so that $x$ is a polynomial in real edges.  Choose $v \in E^0$ such that $xv \neq 0$.  Then $xv$ is a polynomial in only real edges, and the claim holds.

In the inductive step, we assume that for all nonzero $x' \in \A_{N-1}$, there exists $\gamma \in E^*$ such that $x' \gamma \neq 0$ and $x' \gamma$ is a polynomial in only real edges.  Given an element $x = \sum_{i=1}^M r_i \alpha_i \beta_i^* \in \A_N$, we may choose $v \in E^0$ such that $xv \neq 0$.    By regrouping terms, we may write $$xv = \sum_{j=1}^P  x_j e_j^* + y$$ where the $x_j$'s are polynomials in which each term has $N-1$ ghost edges or fewer (so that $x_j \in A_{N-1}$), each $e_j \in E^1$ with $s(e_j) = v$ and $e_j \neq e_{j'}$ for $j \neq j'$, and $y$ is a polynomial in only real edges with $yv = y$.  If $y = 0$, then $xve_1 = x_1 \neq 0$ and by the inductive hypothesis there exists $\gamma'$ such that $x_1 \gamma'$ is a nonzero polynomial in only real edges.  If $\gamma := e_1 \gamma'$, then $x \gamma = xve_1 \gamma' = x_1 \gamma'$ is a nonzero polynomial in only real edges.  

If $y \neq 0$, then we consider three possibilities for $v$.  If $v$ is a regular vertex, then $v = \sum_{s(e) = v} ee^*$ and $xv = \sum_{j=1}^P  x_j e_j^* + \sum_{s(e)=v} yee^*$ and by regrouping we are as in the situation described in the previous paragraph, so we may argue as done there.  If $v$ is a sink, then there are no edges whose source is $v$, so $xv = y$ and we may choose $\gamma := v$, and the claim holds.  If $v$ is an infinite emitter, then we may choose $e \in E^1$ with $s(e) = v$ and $e \neq e_j$ for all $1 \leq j \leq P$.  If we let $\gamma := e$, then $x\gamma = xe = xve = \sum_{j=1}^P  x_j e_j^*e + ye = ye$.  Since $y$ is a nonzero polynomial in only real edges with $yv=y$, it follows from Lemma~\ref{xe-not-zero-lem} that $ye$ is a nonzero polynomial in only real edges.  By the Principle of Mathematical Induction, we may conclude that the claim holds for all $N$, and hence the lemma holds for all nonzero $x \in L_R(E)$.
\end{proof}

\begin{theorem}[Cuntz-Krieger Uniqueness Theorem] \label{CKUT-thm}
Let $E$ be a graph satisfying Condition~(L), and let $R$ be a commutative ring with unit.  If $S$ is a ring and $\phi : L_R(E) \to S$ is a ring homomorphism with the property that  $\phi(rv) \neq 0$ for all $v \in E^0$ and for all $r \in R \setminus \{ 0 \}$, then $\phi$ is injective.
\end{theorem}

\begin{proof}
Suppose $x \in \ker \phi$ and $x \neq 0$.  By Lemma~\ref{shift-to-real-lem} there exists $\gamma \in E^*$ such that $x \gamma$ is a nonzero polynomial in all real edges.  Consequently, Lemma~\ref{shift-real-to-vertex-lem} implies that there exists $\alpha, \beta \in E^*$ such that $\alpha^* x\gamma \beta = rv$ for some $v \in E^0$ and some $r \in R \setminus \{ 0 \}$.  Then $\phi(rv) = \phi(\alpha^*) \phi(x) \phi(\gamma \beta) = 0$, which is a contradiction.  Hence $\ker \phi = \{ 0 \}$, and $\phi$ is injective.
\end{proof}

\begin{corollary}
Let $E$ be a graph satisfying Condition~(L), and let $K$ be a field.  If $S$ is a ring and $\phi : L_K(E) \to S$ is a ring homomorphism with the property that  $\phi(v) \neq 0$ for all $v \in E^0$, then $\phi$ is injective.
\end{corollary}

\begin{remark} \label{CKUT-better-rem}
The Cuntz-Krieger Uniqueness Theorem has a long history in the $C^*$-algebra setting, and the Cuntz-Krieger Uniqueness Theorem for graph $C^*$-algebras may be viewed as a vast generalization of Coburn's Theorem \cite[Theorem~3.5.18]{Mur}.  

The first Cuntz-Krieger Uniqueness Theorem was proven by Cuntz and Krieger in \cite[Theorem~2.13]{CK}, where they showed that if $A$ is a finite $\{0, 1\}$-matrix satisfying Condition~(L) then any two Cuntz-Krieger $A$-families composed of nonzero partial isometries generate isomorphic $C^*$-algebras.  This was generalized to $C^*$-algebras of locally finite graphs in \cite[Theorem~3.7]{KPR} using groupoid techniques, and \cite{KPR} is also where Condition~(L) was first introduced.  In \cite[Theorem~3.1]{BPRS} a Cuntz-Krieger Theorem was proven for $C^*$-algebras of row-finite graphs, and the proof avoided groupoid methods in favor of a direct analysis of the AF-core.  A Cuntz-Krieger Uniqueness Theorem for $C^*$-algebras of non-row-finite graphs was obtained in \cite[Theorem~1.5]{RS} by realizing the graph $C^*$-algebra as an increasing union of $C^*$-algebras of finite graphs.  When Cuntz and Krieger's original theorem \cite[Theorem~2.13]{CK} is translated into a theorem about graphs, one obtains the Cuntz-Krieger Uniqueness Theorem for $C^*$-algebras of finite graphs with no sinks and, moreover, Condition~(I) is equivalent to Conidition~(L) for these graphs.  Additionally, the Cuntz-Krieger Uniqueness Theorem was extended to Cuntz-Krieger algebras of infinite matrices \cite[Theorem~13.1]{EL}.

Although there is a Cuntz-Krieger Uniqueness Theorem for Leavitt path algebras, this result is independent of the graph $C^*$-algebra result --- neither the Cuntz-Krieger Uniqueness Theorem for Leavitt path algebras nor the Cuntz-Krieger Uniqueness Theorem for graph $C^*$-algebras may be used to obtain the other.  In \cite[Corollary~3.3]{AbrPino}, Abrams and Aranda-Pino first proved a weak version of the Cuntz-Krieger Uniqueness Theorem for $L_K(E)$, where $K$ is a field and $E$ is a row-finite graph.  Later, the author proved a lemma (see \cite[Lemma~6.5]{Tom10}) that, with \cite[Corollary~3.3]{AbrPino}, gives a full Cuntz-Krieger Uniqueness Theorem for $L_K(E)$ when $E$ is a row-finite graph.  A proof of the Cuntz-Krieger Uniqueness Theorem for $L_K(E)$, where $K$ is a field and $E$ is an arbitrary graph, was given by the author in \cite[Theorem~6.8]{Tom10}.  The proof in \cite{Tom10} uses the process of desingularization \cite[Lemma~6.7]{Tom10} to show that the Cuntz-Krieger Uniqueness Theorem in the row-finite case implies the Cuntz-Krieger Uniqueness Theorem for arbitrary graphs. The proof in Theorem~\ref{CKUT-thm} uses different techniques than \cite{AbrPino} or \cite{Tom10}, and does not require one to consider the row-finite case first. 
\end{remark}

\section{Ideals in Leavitt path algebras} \label{Ideals-sec}

In this section we analyze ideals in $L_R(E)$.  We will see that for ideals of $L_R(E)$ we will not only be concerned with which vertices are in the ideal, but also which multiples of the vertices are in the ideal.  To motivate the results in this section, we start with an example.  

\begin{example}
Let $E$ be the graph with two vertices and no edges, and let $R = \Z$.  Then $L_\Z(E) \cong \Z \oplus \Z$.  If we consider the ideals of $L_\Z(E)$, we see that they are of the form $n \Z \oplus m \Z$ for $n, m \in \{0, 1, 2, \ldots, \infty \}$.  We would like to consider the ideals that are reflected in the structure of the graph --- in particular, those ideals that are generated by vertices of the graph.  However, if we list the vertices of $E$ as $E^0 = \{v, w \}$, then there are four subsets of vertices, $\emptyset, \{ v \}, \{ w \}, \{v,w \}$, and the ideals generated by these sets are $0, \ \Z \oplus 0, \ 0 \oplus \Z, \ \Z \oplus \Z$.  These are the only ideals generated by subsets of vertices, and each of them has the property that if a nonzero multiple of a vertex in in the ideal, then that vertex is in the ideal.  Consequently, it is only these kind of ideals that will be determined by subsets of vertices in the graph.  This motivates the following definition.
\end{example}

\begin{definition}
Let $R$ be a commutative ring with unit, and let $E$ be a graph.  If $I$ is an ideal in $L_R(E)$, we say that $I$ is \emph{basic} if whenever $r \in R \setminus \{ 0 \}$ and $v \in E^0$, we have $rv \in I$ implies $v \in I$.
\end{definition}

\begin{remark}
Observe that if $K$ is a field, then every ideal in $L_K(E)$ is basic.
\end{remark}

In this section we show that saturated hereditary subsets of vertices correspond to graded basic ideals.  Throughout this section we restrict our attention to the case of row-finite graphs in order to  avoid many of the complications that arise in the non-row-finite case.  Our hope is that this will make our investigations easier for the reader to follow.  Despite this, most of the results in this section do generalize to the non-row-finite setting, provided one uses admissible pairs in place of saturated hereditary subsets.

\begin{definition}
Let $E$ be a graph.  A subset $H \subseteq E^0$ is \emph{hereditary} if for all $e \in E^0$ and $s(e) \in H$ imply that $r(e) \in H$.  A hereditary subset $H$ is \emph{saturated} if whenever $v \in E^0_\textnormal{reg}$ then $r(s^{-1}(v)) \subseteq H$ implies that $v \in H$.  For any hereditary set $X$, we define the \emph{saturation} $\overline{X}$ to be the smallest saturated hereditary subset of $E^0$ containing $X$.
\end{definition}

Observe that intersections of saturated hereditary subsets are saturated hereditary.  Also, unions of saturated hereditary subsets are hereditary, but not necessarily saturated.

In any $R$-algebra $A$, the ideals of $A$ are partially ordered by inclusion and form a lattice under the operations $I \wedge J := I \cap J$ and $I \vee J := I + J$.  (Note that $I + J$ is the smallest ideal containing $I \cup J$.)  This lattice has a maximum element $A$ and a minimum element $\{ 0 \}$.

Likewise, for any graph $E=(E^0,E^1,r,s)$, the saturated hereditary subsets of $E^0$ are partially ordered by inclusion and form a lattice under the operations $H_1 \wedge H_2 := H_1 \cap H_2$ and $H_1 \vee H_2 := \overline{H_1 \cup H_2}$.  This lattice has a maximum element $E^0$ and a minimum element $\emptyset$.

\begin{definition} \label{graph-constructions}
Let $E = (E^0, E^1, r, s)$ be a graph and $H \subseteq E^0$ be a saturated hereditary subset.  We define $(E \setminus H)$ to be the graph with $(E \setminus H)^0 := E^0 \setminus H$,  $(E \setminus H)^1 := E^1 \setminus r^{-1}(H)$, and $r_{(E \setminus H)}$ and $s_{(E \setminus H)}$ are obtained by restricting $r$ and $s$ to $(E \setminus H)^1$.  We also define $E_H$ to be the graph with $E_H^0 := H$, $E_H^1 := s^{-1}(H)$, and $r_{E_H}$ and $s_{E_H}$ are obtained by restricting $r$ and $s$ to $E_H^1$. 
\end{definition}

\begin{lemma} \label{H-I-sat-hered}
Let $E$ be a graph, and let $R$ be a commutative ring with unit.  If $I$ is an ideal of $L_R(E)$, then the set $H_I := \{ v : v \in I \}$ is a saturated hereditary subset.
\end{lemma}

\begin{proof}
If $e \in E^1$ and $s(e) \in H$, then $s(e) \in I$ so $r(e) = e^*e = e^*s(e)e \in I$ and $r(e) \in H$.  Thus $H$ is hereditary.  

If $v \in E^0_\textnormal{reg}$ and $r(s^{-1}(v)) \subseteq H$, then for each $e \in s^{-1}(v)$ we have $r(e) \in H$ and $r(e) \in I$ so $ee^* = er(e)e^* \in I$.  Thus $v = \sum_{s(e) = v} ee^* \in I$, and $v \in H$.  Hence $H$ is saturated.
\end{proof}

\begin{proposition} \label{I-H-facts}
Let $E$ be a graph, and let $R$ be a commutative ring with unit.  If $H$ is a saturated hereditary subset of $E^0$, and $I_H$ is the two-sided ideal in $L_R(E)$ generated by $\{ v : v \in H \}$, then $$I_H = \algspan_R \{ \alpha \beta^* : \alpha, \beta \in E^* \text{ and } r(\alpha) = r(\beta) \in H \},$$ $I_H$ is a graded basic ideal, and $\{ v \in E^0 : v \in I_H \} = H$.  Moreover, $I_H$ is a selfadjoint ideal that is also an idempotent ring.  
\end{proposition}

\begin{proof}
We first observe that the multiplication rules of \eqref{mult-rules-eq} imply that $\algspan_R \{ \alpha \beta^* :  \alpha, \beta \in E^* \text{ and } r(\alpha) = r(\beta) \in H \}$ is a two-sided ideal containing $H$.  It follows that $I_H \subseteq \algspan_R \{ \alpha \beta^* :  \alpha, \beta \in E^* \text{ and } r(\alpha) = r(\beta) \in H \}$.  Furthermore, if $v \in H$, then for any $\alpha, \beta \in E^*$ with $r(\alpha)=r(\beta)=v$, the element $\alpha v \beta^* = \alpha \beta^*$ is in any ideal containing $v$.  Hence $I_H = \algspan_R \{ \alpha \beta^* : \alpha, \beta \in E^* \text{ and } r(\alpha) = r(\beta) \in H \}$.

To see that $I_H$ is graded it suffices to notice that $\alpha \beta^*$ is homogeneous of degree $|\alpha|-|\beta|$.  In addition, we see $I_H$ is selfadjoint because $(\alpha\beta^*) = \beta \alpha^*$.  Next we show that $I_H$ is a basic ideal.  Let $v \in E^0$ and suppose that $rv \in I_H$ for some $r \in R \setminus \{ 0 \}$.  Let $E \setminus H$ be the graph of Definition~\ref{graph-constructions}.  Then the vertices, edges, and ghost edges of $E \setminus H$, which generate $L_R(E \setminus H)$, may be extended to a Leavitt $E$-family by simply defining elements to be zero if $v \in H$ or $r(e) \in H$.  By the universal property of $L_R(E)$, we obtain an $R$-algebra homomorphism $\phi : L_R(E) \to L_R (E \setminus H)$ with $$\phi (v) = \begin{cases} v & \text{if $v \in E^0 \setminus H$} \\ 0 & \text{if $v \in H$} \end{cases} \qquad \phi (e) = \begin{cases} e & \text{if $r(e) \in E^0 \setminus H$} \\ 0 & \text{if $r(e) \in H$} \end{cases}$$ and $$\phi (e^*) = \begin{cases} e^* & \text{if $r(e) \in E^0 \setminus H$} \\ 0 & \text{if $r(e) \in H$.} \end{cases}$$
Thus $\ker \phi$ is a two-sided ideal of $L_R(E)$ containing $H$, and it follows that $I_H \subseteq \ker \phi$.  Hence $r \phi(v) = \phi (rv) = 0$, and since $v$ is a vertex in $E^0$, either $\phi(v) = v$ or $\phi( v ) = 0$.  But Proposition~\ref{LPA-exists} implies that in $L_R(E \setminus H)$ we have $rv \neq 0$ for all $v \in (E \setminus H)^0$ and all $r \in R \setminus \{ 0 \}$.  Thus $\phi( v ) = 0$ and $v \in H$. Hence $v \in I_H$, and $I_H$ is a basic ideal.

We next show that the set $\{ v \in E^0 : v \in I_H \}$ is precisely $H$.  To begin, we trivially have $H \subseteq \{ v \in E^0 : v \in I_H \}$.  For the reverse inclusion we use the fact that $I_H \subseteq \ker \phi$ to conclude that $v \notin H$ implies that $\phi(v) \neq 0$ so that $v \notin \ker \phi$ and $v \notin I_H$.  Hence $\{ v \in E^0 : v \in I_H \} = H$.

Finally we show that $I_H$ is an idempotent ring.  Any $x \in I_H$ has the form $x = \sum_{i=1}^N r_i \alpha_i \beta_i^*$ with $r(\alpha_i) = r(\beta_i) \in H$.  For each $i$, define $v_i := r(\alpha_i) = r(\beta_i)$.  Then $r_i\alpha_i \beta_i^* = (r_i \alpha_i v_i) (v_i \beta_i^*)$, and since $r_i \alpha_i v_i \in I_H$ and $v_i \beta_i^* \in I_H$, we see that any $x \in I_H$ may be written as $x = a_1b_1 + \ldots + a_Nb_N$ for $a_1, \ldots, a_N, b_1, \ldots, b_N \in I_H$.  Thus $I_H$ is an idempotent ring.
\end{proof}

\begin{lemma} \label{hered-sat-ideal}
Let $E$ be a graph, and let $R$ be a commutative ring with unit.  If $X$ is a hereditary subset of $E^0$, and $I_X$ is the two-sided ideal in $L_R(E)$ generated by $\{ v : v \in X \}$, then $$I_X = I_{\overline{X}}.$$  In particular, $I_X$ is a graded basic ideal that is also an idempotent ring.  
\end{lemma}

\begin{proof}
Since $X \subseteq \overline{X}$, we have $I_X \subseteq I_{\overline{X}}$.  Conversely, if we let $H := \{ v \in E^0 : v \in I_X \}$, then it follows from Lemma~\ref{H-I-sat-hered} that $H$ is a saturated hereditary subset containing $X$.  Thus $\overline{X} \subseteq H$, and $v \in \overline{X}$ implies $v \in I_X$.  Hence $I_{\overline{X}} \subseteq I_X$.
\end{proof}

\begin{theorem} \label{graded-ideals-structure}
Let $E$ be a graph, and let $R$ be a commutative ring with unit. Using the notation of Definition~\ref{graph-constructions} and Proposition~\ref{I-H-facts}, we have the following:
\begin{itemize}
\item[(1)]  The map $H \mapsto I_H$ is a lattice isomorphism from the lattice of saturated hereditary subsets of $E^0$ onto the lattice of graded basic ideals of $L_R(E)$.  In particular, the graded basic ideals of $L_R(E)$ form a lattice with $$I_{H_1} \wedge I_{H_2} = I_{H_1 \cap H_2} \qquad \text{ and } \qquad I_{H_1} \vee I_{H_2} = I_{\overline{H_1 \cup H_2}}$$ for any saturated hereditary subsets $H_1$ and $H_2$.
\item[(2)] For any saturated hereditary subset $H$ we have that $L_R(E) / I_H$ is canonically isomorphic to $L_R(E\setminus H)$.
\item[(3)]  For any hereditary subset $X$ the ideal $I_X$ and the Leavitt path algebra $L_R (E_X)$ are Morita equivalent as rings.
\end{itemize}
\end{theorem}

\begin{proof}
We shall first prove (2), then (1), and then (3).

\smallskip

\noindent \underline{\textsc{Proof of (2):}}  We shall show that $L_R(E) / I_H \cong L_R(E \setminus H)$.  Let $\{v : v \in E^0 \} \cup \{e, e^* : \in E^1\}$ be the generators for $L_R(E)$.  Then $\{ v + I_H : v \in E \setminus H \} \cup e+I_H, e^*+I_H : r(e) \notin H\}$ is a collection of elements satisfying the Leavitt path algebra relations for $E_H$ and generating $L_R(E)/I_H$.  Hence there exists a surjective $R$-algebra homomorphism $\phi : L_R(E_H) \to L_R(E)/I_H$.  Proposition~\ref{I-H-facts} shows that $I_H$ is a graded ideal, and hence $\phi$ is a  graded homomorphism.  Furthermore, if $v \in E^0_H$, then $v \notin H$ and the previous paragraph implies that $v \notin I_H$.  Since Proposition~\ref{I-H-facts} shows that $I_H$ is a basic ideal, for all $v \in E_H^0$ and all $r \in R \setminus \{ 0 \}$, we have $\phi (rv) = rv + I_H \neq 0$.  It follows from the Graded Uniqueness Theorem~\ref{GUT} that $\phi$ is injective.  Thus $\phi$ is an isomorphism and $L_R(E) / I_H \cong L_R(E \setminus H)$.

\smallskip

\noindent \underline{\textsc{Proof of (1):}}  We shall show that $H \mapsto I_H$ is a lattice
isomorphism.  To see that this map is surjective, let $I$ be a
graded basic ideal in $L_R(E)$, and set $H := \{v \in E^0
: v \in I\}$.  Since $I_H \subseteq I$, we see that $I_H$ and $I$
contain the same $v$'s.  Therefore, just as
in the proof of Part (2), we see that $L_R(E) /
I_H$ and $L_R(E) / I$ are generated by nonzero elements satisfying the Leavitt path algebra relations for 
$E \setminus H$.  Since both $I_H$ and $I$ are graded, both
quotients are graded, and the quotient map $\pi : L_R(E) / I_H \rightarrow L_R(E) / I$ is a graded homomorphism.  Furthermore, since $I$ and $I_H$ contain the same $v$'s, and since $I$ is a basic ideal, it follows that if $v \in E^0 \setminus H$, then $v \notin I_H$ and $rv \notin I$ for all $r \in R \setminus \{0 \}$.  Thus the Graded Uniqueness Theorem implies that the quotient map $\pi : L_R(E \setminus H) \cong
L_R(E) / I_H \rightarrow L_R(E) / I$ is injective.  Hence $I = I_H$.  

The fact that $H \mapsto I_H$ is injective
follows immediately from the fact that $\{ v \in E^0 : v \in I_H \}$ is precisely $H$, which was obtained in Proposition~\ref{I-H-facts}.  Thus the correspondence $H \mapsto I_H$ is bijective.  Since $H \mapsto I_H$ is a bijection that preserves inclusions, the map $H \mapsto I_H$ is a poset isomorphism and hence automatically a lattice isomorphism

\smallskip 

\noindent \underline{\textsc{Proof of (3):}}  

To see that $I_X$ is Morita equivalent to $L_R(E_X)$, list the elements of $X = \{ v_ 1, v_2, \ldots \}$, let $$\Lambda := \begin{cases} \{1, 2, \ldots, |X| \} & \text{ if $X$ is finite} \\ \{ 1, 2, \ldots \} & \text{ if $X$ is infinite,} \end{cases}$$
and let $e_n := \sum_{i=1}^n v_i$ for $n \in \Lambda$.  

If we consider the elements $\{ v : v \in H \}$ and $\{ e, e^* : e \in E^1 \text{ and } s(e) \in H\}$ in $L_R(E)$, we see that they are a Leavitt $E_X$-family and thus there exists a homomorphism $\pi : L_R (E_X) \to L_R(E)$ taking the generators of $L_R(E_X)$ to these elements.  Since this homomorphism is graded, Theorem~\ref{GUT} shows that $\pi$ is injective.  Hence we may identify $L_R(E_X)$ with the subalgebra $$\algspan_R \{ \alpha \beta^* : \alpha, \beta \in E_X^* \text{ and } r(\alpha)=r(\beta) \in X  \}$$ of $L_R(E)$.  With this identification, we see that $L_R( E_X) = \sum_{n=1}^\infty e_n L_R(E) e_n$.  Moreover, Lemma~\ref{I-H-facts} shows that $I_X = \sum_{n=1}^\infty L_R(E) e_n L_R(E)$. 

In addition,
$$\left(\sum_{n \in \Lambda} e_n L_R(E) e_n, \sum_{n \in \Lambda} L_R(E) e_n L_R(E), \sum_{n \in \Lambda} L_R(E)e_n, \sum_{n \in \Lambda} e_n L_R(E), \psi, \phi \right)$$ with $\psi(m \otimes n) = mn$ and $\phi(n \otimes m) =nm$ is a (surjective) Morita context for the idempotent rings $L_R(E_X)$ and $I_X$.  It then follows from \cite[Proposition~2.5]{GS} and \cite[Proposition~2.7]{GS} that $L_R(E_X)$ and $I_X$ are Morita equivalent.

\end{proof}

\begin{corollary}
Let $E$ be a graph, and let $R$ be a commutative ring with unit. Then every graded basic ideal of $L_R(E)$ is selfadjoint.
\end{corollary}

Using the Cuntz-Krieger Uniqueness Theorem we can characterize those graphs whose associated Leavitt path algebras have the property that every basic ideal is a graded ideal.

\begin{definition}
We say that a closed path $\alpha = e_1 \ldots e_n \in E^n$ is \emph{simple} if $s(e_i) \neq s(e_1)$ for $i = 2, 3, \ldots, n$.
\end{definition}

\begin{definition}
A graph $E$ satisfies \emph{Condition~(K)} if every vertex in $E^0$ is either the base of no closed path or the base of at least two simple closed paths.
\end{definition}

The following proposition is well known.  It has been proven in \cite[Proposition~1.17]{Tom9} and \cite[Theorem~4.5(2),(3)]{PinParMol}.

\begin{proposition} \label{K-implies-quotient-L}
If $E$ is a row-finite graph, then $E$ satisfies Condition~(K) if and only if for every saturated hereditary subset $H$, the graph $E \setminus H$ of Definition~\ref{graph-constructions} satisfies Condition~(L).
\end{proposition}

\begin{lemma} \label{closed-path-M-Kx}
If $E$ is the graph consisting of a single simple closed path of length $n$; i.e., 
$$
E^0 = \{v_1, \ldots, v_n\} \quad E^1 = \{e_1, \ldots e_n \} $$ 
$$s(e_i) = v_i \quad \text{ for $1 \leq i \leq n$}$$ 
$$r(e_i) = v_{i+1} \quad \text{ for $1 \leq i < n$} \quad \text{ and } \quad r(e_n) = v_1,$$ an d $R$ is a commutative ring with unit, then $L_R(E) \cong M_n(R[x,x^{-1}])$.
\end{lemma}

The proof of Lemma~\ref{closed-path-M-Kx} is the same as the proof of \cite[Lemma~6.12]{Tom10}.

\begin{lemma} \label{ideals-tran-lem}
Let $R$ be a commutative ring with unit, let $E$ be a row-finite graph, and let $H$ be a saturated hereditary subset of $E$.  Then the ideal $I_H$ in $L_R(E)$ is a ring with a set of local units.
\end{lemma}

The proof of Lemma~\ref{ideals-tran-lem} is the same as the proof of \cite[Lemma~6.14]{Tom10}.

\begin{lemma} \label{not-L-non-graded-ideals}
Let $R$ be a commutative ring with unit, and let $E$ be a row-finite graph that contains a simple closed path with no exit.  Then $L_R(E)$ contains an ideal that is basic but not graded.  
\end{lemma}

\begin{proof}
Let $\alpha := e_1 \ldots e_n$ be a simple closed path with no exits in $E$.  If we let $X : = \{s(e_i) \}_{i=1}^n$, then since $\alpha$ has no exits, $X$ is a hereditary subset of $E^0$.  By Theorem~\ref{graded-ideals-structure}(3) $L_R(E_X)$ is Morita equivalent to the ideal $I_X$ in $L_R(E)$.  However, $E_X$ is the graph which consists of a single closed path, and thus $L_R(E_X) \cong M_n(R[x,x^{-1}])$ by Lemma~\ref{closed-path-M-Kx}.  Theorem~\ref{graded-ideals-structure}(1) implies that $L_R(E) \cong M_n(R[x,x^{-1}])$ has no proper nontrivial graded ideals.  Let $I := \langle x+1 \rangle$ be the ideal in $R[x, x^{-1}]$ generated by $x+1$.  Then any element of $I$ has the form $p(x) (x+1)$ for some $p(x) \in R[x,x^{-1}]$ and hence has $-1$ as a root.  It follows that for every $r \in R \setminus \{ 0 \}$ we have that $r1 \notin I$.  Since $v = 1$ in $R[x, x^{-1}]$, it follows that $rv \notin I$ for all  $r \in R \setminus \{ 0 \}$.  Thus $I$ is a basic ideal.  It follows that $M_n(I)$ is a proper nontrivial ideal of  $M_n( R[x, x^{-1}])$, which is basic but not graded.  Because the Morita context described in the proof of Theorem~\ref{graded-ideals-structure}(3) gives a lattice isomorphism from ideals of $L_R(E_X)$ to ideals of $I_X$ that preserves the grading, we may conclude that $I_X$ contains an ideal that is basic but not graded.  Since $I_X$ has a set of local units by Lemma~\ref{ideals-tran-lem}, it follows from Lemma~\ref{local-units-imply-trans} that ideals of $I_X$ are ideals of $L_R(E)$.  Hence $L_R(E)$ contains an ideal that is basic but not graded.  
\end{proof}

These results together with the Cuntz-Krieger Uniqueness Theorem give us the following theorem.

\begin{theorem} \label{K-iff-ideals-graded}
Let $R$ be a commutative ring with unit.  If $E$ is a row-finite graph, then $E$ satisfies Condition~(K) if and only if every basic ideal in $L_R(E)$ is graded.
\end{theorem}

\begin{proof}
Suppose that $E$ satisfies Condition~(K).  If $I$ is a basic ideal of $L_R(E)$, let $H := \{ v : v \in I \}$.  Then $I_H \subseteq I$, and we have a canonical surjection $q : L_R(E) / I_H \to L_R(E) / I$.  By Theorem~\ref{graded-ideals-structure}(2) there exists a canonical isomorphism $\phi : L_R(E \setminus H) \to L_R(E) / I_H$.  Since $I$ is basic, the composition $q \circ \phi : L_R(E \setminus H) \to L_R(E) / I$ has the property that $(q \circ \phi) (rv) \neq 0$ for all $v \in E^0$ and $r \in R \setminus \{ 0 \}$.  Since $E$ satisfies Condition~(K), it follows from Proposition~\ref{K-implies-quotient-L} that $E \setminus H$ satisfies Condition~(L).  Hence we may apply Theorem~\ref{CKUT-thm} to conclude that $q \circ \phi$ is injective.  Since $\phi$ is an isomorphism, this implies that $q$ is injective and $I = I_H$.  It then follows from Lemma~\ref{I-H-facts} that $I$ is graded.

Conversely, suppose that $E$ does not satisfy Condition~(K).  Then Proposition~\ref{K-implies-quotient-L} implies that there exists a saturated hereditary subset $H$ such that $E \setminus H$ does not satisfy Condition~(L).  Thus there exists a closed simple path with no exit in $E \setminus H$, and by Lemma~\ref{not-L-non-graded-ideals} the algebra $L_R(E \setminus H) \cong L_R(E)/ I_H$ contains an ideal $I$ that is basic and not graded.  If we let $q : L_R(E) \to L_R(E)/ I_H$ be the quotient map, then $q$ is graded and $q^{-1}(I)$ is an ideal of $L_R(E)$ that is basic but not graded.
\end{proof}

\begin{corollary}
If $E$ is a row-finite graph that satisfies Condition~(K), then the map $H \mapsto I_H$ is a lattice isomorphism from the lattice of saturated hereditary subsets of $E$ onto the lattice of basic ideals of $L_R(E)$.  
\end{corollary}

\begin{definition}
The Leavitt path algebra $L_R(E)$ is \emph{basically simple} if the only basic ideals of $L_R(E)$ are $\{ 0 \}$ and $L_R(E)$.  (Note that if $R = K$ is a field, then $L_K(E)$ is basically simple if and only if $L_K(E)$ is simple.)
\end{definition}

\begin{theorem} \label{bas-simp-cond-thm}
Let $R$ be a commutative ring with unit, and let $E$ be a row-finite graph.  The Leavitt path algebra $L_R(E)$ is basically simple if and only if $E$ satisfies both of the following  conditions:
\begin{itemize}
\item[(i)] The only saturated hereditary subsets of $E$ are $\emptyset$ and $E^0$, and
\item[(ii)] The graph $E$ satisfies Condition~(L).
\end{itemize}
\end{theorem}

\begin{proof}
Suppose that $L_R(E)$ is basically simple.  Then the only basic ideals of $L_R(E)$ are $\{0\}$ and $L_R(E)$, both of which are graded.  By Theorem~\ref{K-iff-ideals-graded} we have that $E$ satisfies Condition~(K).  It then follows from Theorem~\ref{graded-ideals-structure}(1) and the fact that $L_R(E)$ is basically simple, that the only saturated hereditary subsets of $E$ are $\emptyset$ and $E^0$.  Hence (i) holds.  In addition, since Condition~(K) implies Condition~(L) (cf.~Proposition~\ref{K-implies-quotient-L}) we have that (ii) holds.

Conversely, suppose that (i) and (ii) hold.  We shall show that $E$ satisfies Condition~(K).  Let $v$ be a vertex and let $\alpha = e_1 \ldots e_n$ be a closed simple path based at $v$.  By (ii) we know that $\alpha$ has an exit $f$; i.e., there exists $f \in E^1$ with $s(f) = s(e_i)$ and $f \neq e_i$ for some $i$.  If we let $H$ be the set of vertices in $E^0$ such that there is no path from that vertex to $v$, then $H$ is saturated hereditary.  By (i) we must have either $H = \emptyset$ or $H = E^0$.  Since $v \notin H$, we have $H = \emptyset$.  Hence for every vertex in $E^0$, there is a path from that vertex to $v$.  Choose a path $\beta \in E^*$ from $r(f)$ to $v$ of minimal length.  Then $e_1 \ldots e_{i-1} f \beta$ is a simple closed path based at $v$ that is distinct from $\alpha$.  Hence $E$ satisfies Condition~(K).  It then follows from Theorem~\ref{graded-ideals-structure}(1) and (i) that $L_R(E)$ is basically simple.
\end{proof}

Condition (i) and (ii) in the above theorem can be reformulated in a number of equivalent ways.  The equivalence of the statements (2)--(5) in Proposition~\ref{simple-equiv-prop} are elementary facts about directed graphs (cf.~\cite[Theorem~1.23]{Tom9} and \cite[Proposition~3.2]{AbrPino3}).

\begin{definition}
A graph $E$ is \emph{cofinal} if whenever $e_1 e_2 e_3 \ldots$ is an infinite path in $E$ and $v \in E^0$, then there exists a finite path from $v$ to $s(e_i)$ for some $i \in \N$.
\end{definition}

\begin{proposition} \label{simple-equiv-prop}
Let $E$ be a row-finite graph, let $R$ be a commutative ring with unit, and let $L_R(E)$ be the associated Leavitt path algebra.  Then the following are equivalent.
\begin{enumerate}
\item $L_R(E)$ is basically simple.
\item $E$ satisfies Condition~(L), and the only saturated hereditary subsets of $E^0$ are $\emptyset$ and $E^0$.
\item $E$ satisfies Condition~(K), and the only saturated hereditary subsets of $E^0$ are $\emptyset$ and $E^0$.
\item $E$ satisfies Condition~(L), $E$ is cofinal, and whenever $v$ is a sink in $E$ and $w \in E^0$ there is a path from $w$ to $v$.
\item  $E$ satisfies Condition~(K), $E$ is cofinal, and whenever $v$ is a sink in $E$ and $w \in E^0$ there is a path from $w$ to $v$.
\end{enumerate}
\end{proposition}

\begin{remark} \label{ideal-history}
The techniques used in this section are similar to those used to analyze ideals of graph $C^*$-algebras, which were inspired by the work of Cuntz and Krieger.  In \cite{CK} Cuntz and Krieger showed that the Cuntz-Krieger algebra of an irreducible matrix satisfying Condition~(I) is simple.  In \cite[Theorem~2.5]{C2} Cuntz showed that for the Cuntz-Krieger algebra of a matrix satisfying Condition~(II) there is a bijective correspondence between the ideals of the Cuntz-Krieger algebra and the hereditary subsets of a certain finite partially ordered set associated with the matrix.  Subsequently, it was shown in \cite[Theorem~3.5]{aHR} that the hereditary subsets of this partially ordered set correspond to the gauge-invariant ideals in any universal Cuntz-Krieger algebra of a finite $\{0,1\}$-matrix.  

In \cite[Theorem~6.6]{KPRR}, the authors introduced Condition~(K) for graphs, and showed that for a locally finite graph satisfying Condition~(K) there is a bijective correspondence between ideals in the graph $C^*$-algebra and saturated hereditary subsets of the graph.  Their proof used groupoid techniques and relied on realizing the graph $C^*$-algebra as the $C^*$-algebra of a groupoid.  In \cite[Theorem~4.1]{BPRS} it was shown that for $C^*$-algebras of row-finite graphs there is a bijective correspondence between gauge-invariant ideals in the graph $C^*$-algebra and saturated hereditary subsets of the graph, and in \cite[Theorem~4.4]{BPRS} it is proven that when a graph satisfies Condition~(K) all the ideals of the associated $C^*$-algebra are gauge invariant.  The techniques used in \cite{BPRS} avoided the use of groupoids, and instead used methods similar to those used by Cuntz in \cite{C2}.  In \cite[Theorem~3.6]{BHRS} and \cite[Theorem~3.5]{DT1} the analysis of ideals was extended to non-row-finite graphs, where new phenomena had to be accounted for, and it was shown that gauge-invariant ideals of the graph $C^*$-algebra are in bijective correspondence with \emph{admissible pairs}; i.e., pairs consisting of a saturated hereditary set and a subset of breaking vertices for this saturated hereditary subset.  Furthermore, these results have been generalized to Cuntz-Pimsner algebras, and it has been shown that the gauge-invariant ideals in a Cuntz-Pimsner algebra correspond to admissible pairs of ideals in the coefficient algebra of the Hilbert bimodule \cite[Theorem~8.6]{Kat5}.

In the past five years, methods similar to those in \cite{CK}, \cite{C2}, \cite{aHR}, \cite{BPRS}, \cite{BHRS}, and \cite{DT1} have been used to analyze the ideal structure of Leavitt path algebras over fields.  It has been shown in \cite[Theorem~3.11]{AbrPino} that $L_K(E)$ is simple if and only if $E$ satisfies Condition~(L) and the only saturated hereditary subsets of $E$ are $\emptyset$ and $E^0$.  In \cite[Theorem~5.3]{AMP} it was shown that if $E$ is a row-finite graph, then the graded ideals of $L_K(E)$ are in bijective correspondence with the saturated hereditary subsets of $E$.  Furthermore, in \cite[Theorem~5.7]{Tom10} it was shown that for a non-row-finite graph $E$ the graded ideals of $L_K(E)$ are in bijective correspondence with the admissible pairs of $E$.  Moreover, it was proven in \cite[Theorem~6.16]{Tom10} that a graph $E$ satisfies Condition~(K) if and only if every ideal in the Leavitt path algebra $L_K(E)$ is graded.
\end{remark}

\section{Tensor products and changing coefficients} \label{tensor-sec}

If $R$ is a commutative ring with unit and if $A$ and $B$ are $R$-algebras, then the tensor product $A \otimes_R B$ is an $R$-module that may also be given the structure of an $R$-algebra with a multiplication satisfying $(a_1 \otimes b_1) (a_2 \otimes b_2) = a_1 a_2 \otimes b_1 b_2$.  (See \cite[Ch.IV, Theorem~7.4]{Hun} for details on how this multiplication is obtained.) Furthermore, if $R$ is a commutative ring with unit that contains a unital subring $S$, then we may view $R$ as an $S$-algebra.  If, in addition, $A$ is an $S$-algebra, then $R \otimes_S A$ is an $R$-algebra with $r_1 (r_2 \otimes a) = r_1r_2 \otimes a$.

\begin{theorem} \label{change-coefficients}
Let $R$ be an algebra over the commutative unital ring $S$, and let $E$ be a graph.  Then $$L_R(E) \cong R \otimes_S L_S(E)$$ as $R$-algebras.
\end{theorem}

\begin{proof}
One can verify that $$\{ 1 \otimes v : v \in E^0 \} \cup \{ 1 \otimes e, 1 \otimes e^* : e \in E^1 \}$$ is a Leavitt $E$-family in the $R$-algebra $R \otimes_S L_S(E)$, and hence there exists an $R$-algebra homomorphism $\phi : L_R(E) \to R \otimes_S L_S(E)$ with $\phi(v) = 1\otimes v$, $\phi(e) = 1 \otimes e$, and $\phi(e^*) = 1 \otimes e^*$.  Furthermore, $L_R(E)$ is an $S$-algebra that contains a Leavitt $E$-family $\{v : v \in E^0 \} \cup \{ e, e^* : e \in E^1 \}$.  Thus there exists an $S$-algebra homomorphism $\phi : L_S(E) \to L_R(E)$ with $\phi(v) = v$, $\phi(e)=e$, and $\phi(e^*) = e^*$.  Furthermore, using the universal property of the tensor product, one can verify that there exists an $R$-module morphism $\psi : R \otimes_S L_S(E) \to L_R(E)$ with $\psi (r \otimes x) = r \phi(x)$.  Finally, one can verify that $\psi$ is an inverse for $\phi$ (simply check on generators), and hence $\phi$ is an $R$-algebra isomorphism. 
\end{proof}

\begin{corollary} \label{fund-ring-K-lem}
Let $E$ be a graph, and let $K$ be a field.  If we view $K$ as a $\Z$-module, then $$L_K(E) \cong K \otimes_\Z L_{\Z} (E).$$ Furthermore, if $K$ has characteristic $p > 0$, then we may view $K$ as a $\Z_p$-module and $$L_K(E) \cong K \otimes_{\Z_p} L_{\Z_p} (E).$$
(Here $L_\Z(E)$ denotes the Leavitt path ring of characteristic $0$ associated to $E$, and $L_{\Z_p} (E)$ denotes the Leavitt path ring of characteristic $p$ associated to $E$, as described in Definition~\ref{fundamental-ring-def}.)
\end{corollary}

Let $R$ be a commutative ring with unit that contains a unital subring $S$, and let $E$ be a row-finite graph.  For a saturated hereditary subset $H$ of $E$, let $I_H^S$ denote the ideal in $L_S(E)$ generated by $\{ v : v \in H \}$ and let $I_H^R$ denote the ideal in $L_R(E)$ generated by $\{ v : v \in H \}$.  Theorem~\ref{graded-ideals-structure} shows that any graded basic ideal of $L_S(E)$ has the form $I_H^S$, and any graded basic ideal of $L_R(E)$ has the form $I_H^R$.  Thus the map $I_H^S \mapsto I_H^R$ is a lattice isomorphism from the lattice of graded basic ideals of $L_S(E)$ onto the lattice of graded basic ideals of $L_R(E)$.  If we use Theorem~\ref{change-coefficients} to identify $L_R(E)$ with $R \otimes_S L_R(E)$ via the isomorphism described in the proof, then $I_H^R = R \otimes I_H^S$, and we see that $I \mapsto R \otimes I$ is a map from ideals of $L_S(E)$ onto ideals of $L_R(E)$ that restricts to an isomorphism from graded basic ideals of $L_S(E)$ onto graded basic ideal of $L_R(E)$.  In the special case that $S = \Z$ and $R = K$ is a field (respectively, a field of characteristic $p$), we see that all ideals of $L_K(E)$ are basic, and hence the map $I \mapsto K \otimes I$ is a lattice isomorphism from the lattice of graded basic ideals of $L_\Z(E)$ (respectively, $L_{\Z_p}(E)$) onto the lattice of graded ideals of $L_K(E)$.  This suggests that properties of graded ideals of $L_K(E)$ may derived from properties of graded basic ideals of $L_\Z(E)$ and $L_{\Z_n}(E)$.  

In the study of Leavitt path algebras over fields, it has frequently been found that properties of $L_K(E)$ depend only on properties of the graph $E$ and are independent of the particular field $K$ that is chosen.  The fact that $L_K(E) \cong K \otimes_\Z L_\Z(E)$ (and $L_K(E) \cong K \otimes_{\Z_p} L_{\Z_p} (E)$ if $\ch K = p$), suggests that properties of $L_K(E)$ may consequences of properties of the Leavitt path rings $L_\Z(E)$ and $L_{\Z_p}(E)$.  One may speculate that this is the reason many properties of $L_K(E)$ are independent of $K$.

\end{document}